\documentclass[12pt]{article}

\usepackage[T1]{fontenc}
\usepackage[utf8]{inputenc}
\usepackage{amsmath,amsthm}

\usepackage{newtxtext,newtxmath} 

\usepackage[top=2.5cm,bottom=2.5cm,left=2cm,right=2cm]{geometry}
\usepackage{setspace}
\usepackage{graphicx}
\PassOptionsToPackage{pdftex}{graphicx}
\usepackage{epstopdf-base}
\usepackage{caption}
\usepackage{subcaption}

\usepackage{enumitem}

\usepackage[hidelinks]{hyperref}
\usepackage{url}

\usepackage{todonotes}

\theoremstyle{plain} 
\newtheorem{theorem}{Theorem}
\newtheorem{lemma}[theorem]{Lemma}
\newtheorem{proposition}[theorem]{Proposition}

\theoremstyle{definition} 

\theoremstyle{remark} 

\begin{document}
\title{Quadratic Convergence of a Projection Method for a Plane Curve Feasibility Problem}

\author{
Jordan Collard\thanks{Centre for Optimisation and Decision Science, Curtin University, Perth, WA, Australia.
Emails: \texttt{jordan.collard@postgrad.curtin.edu.au}, \texttt{scott.lindstrom@curtin.edu.au}}
\and
Scott B. Lindstrom\footnotemark[1]
}

\date{} 

%
%
\maketitle

\abstract{Under conditions that prevent tangential intersection, we prove quadratic convergence of a projection algorithm for the feasibility problem of finding a point in the intersection of a smooth curve and line in $\mathbb{R}^2$. This nonconvex problem has been studied in the literature for both Douglas--Rachford algorithm (DR) and circumcentered reflection method (CRM), because it is prototypical of inverse problems in signal processing and image recovery. 
This result highlights the potential of extrapolated methods to meaningfully accelerate convergence in structured feasibility problems. Numerical experiments confirm the theoretical findings. Our work lays the foundations for extending such results to higher dimensional problems.}

\section{Introduction}
In a Hilbert space $\mathcal{H}$, the $r$-sets feasibility problem is to find a point within the  intersection of $r$ finite closed sets, \(C_1, \dots, C_r \subseteq \mathcal{H}\). 
Many optimisation tasks, including inverse problems \cite{artacho2014douglas}, image reconstruction \cite{bauschke2002phase}, signal recovery \cite{combettes2007douglas} and even combinatorial problems \cite{schaad2010modeling}, can be reformulated as feasibility problems. In these cases, the goal is to find a point in the intersection of constraint sets that model the problem's structure.
 
Projections methods are commonly utilised to solve feasibility problems as they break down hard problems into simpler subproblems involving projections onto individual sets. Instead of attempting to compute the intersection directly, we apply successive projections to iteratively approach a feasible point. 

Among projection-based methods, the Douglas--Rachford (DR) algorithm is notable for its strong theoretical guarantees in convex settings, and has shown practical effectiveness for nonconvex problems such as phase retrieval \cite{bauschke2002phase} and graph colouring \cite{aragon2018solving}. A survey of the method's development and applications over the past six decades can be found in \cite{lindstrom2021survey}, and includes the history of its study in the present context (plane curves). 

The Douglas--Rachford method has been extensively studied for the class of plane curve problems we consider here (see, for example, \cite{aragon2013global, borwein2018dynamics, borwein2011douglas,  lindstrom2017computing}), and converges linearly. For many problems, and for all the problems we consider here, the iterates of the Douglas--Rachford method exhibit a spiralling behaviour as they converge to a fixed point (which, if not already a solution to the problem, readily yields one). So motivated, researchers have considered the sequential iterates as a discrete time dynamical system, and brought to bear tools from dynamical systems theory. For problems from the class we consider here, that spiralling phenomenon motivated the construction of explicit Lyapunov functions that describe the local convergence \cite{benoist2015sphere,dao2019lyapunov}. The zeros of the Lyapunov function correspond to the fixed points of the operator that defines the Douglas--Rachford dynamical system (i.e. the Douglas--Rachford operator). Motivated by these known constructions, a class of meta-algorithms called Lyapunov surrogate methods was introduced in \cite{lindstrom2022computable}. These methods aim to approximate the centre of this spiral (i.e. the limiting fixed point), by minimising spherical functions that are intended to serve as surrogates of the Lyapunov function. It was further shown in \cite{lindstrom2022computable} that, for plane curve problems like those we consider, the circumcentered reflection method (CRM, see, for example, \cite{arefidamghani2021circumcentered, circumcentering, behling2018linear, behling2019convex}) may be characterised as such a method.

The Lyapunov surrogate method is currently state-of-the-art for the feasibility problem of finding compactly supported smooth orthogonal wavelets with symmetry or cardinality constraints \cite{dizon2022centering}. It has also been benchmarked on basis pursuit problems \cite{calcan2024admm}, which belong to a more general class than feasibility, where it performed well for randomised data, but exhibited limited improvement when applied to real audio. 


Recently, \cite{dizon2022circumcentering} showed that for the class of problems we consider here, CRM enjoys local quadratic convergence. That proof involved tying CRM to Newton--Raphson method for finding roots of functions on $\mathbb{R}$, and does not extend in any apparent way to analysing the Lyapunov surrogate method. Notwithstanding, it was included for comparison in the numerical experiments on spheres and non-tangential hyperplanes (which reduce to 2-dimensional plane curve problems by symmetry), and these experiments suggested that Lyapunov surrogate methods similarly enjoy a quadratic rate of convergence for problems in this class. In this paper, we provide the proof of that local quadratic convergence. The proof is entirely new and makes explicit use of the relationship that the method has with the existing known Lyapunov functions. We also present numerical results that illustrate the realised quadratic rate.

\section{Preliminaries}
Let $\mathcal{H}$ be some Hilbert space with inner product $\langle\cdot,\cdot\rangle$ and induced norm $\|\cdot\|$. Let $\mathcal{B}_r(y)$ be an open ball about $y$ of radius $r>0$. Suppose $C$ is a nonempty set in $\mathcal{H}$.
The projection of $x$ onto a nonempty $C\subseteq \mathcal{H}$ for all $x\in \mathcal{H}$ is defined by:
\begin{equation*}
     \mathbf{P}_C(x) := \Big\{p\in C: \|p-x\| =\underset{c\in C}{\inf}\|c-x\| \Big\}.
\end{equation*}
When $C$ is nonconvex, $\mathbf{P}_C$ can be a set-valued map that may be empty or contain more than one point. We are only interested in the case when $C$ is nonempty, so we work with a selector for $P_C$, that is a map 
\begin{equation*}
    P_C:\mathcal{H}\rightarrow C: x \mapsto P_C(x)\in \mathbf{P}_C(x),\quad \text{ so }P_C^2=P_C.
\end{equation*}

%

The reflection of $x\in \mathcal{H}$ about $C$ is defined as $R_C:= 2P_C-I,$ where $I$ is the identity operator.


\subsection{Douglas--Rachford Method}
The Douglas--Rachford method was initially introduced in relation to nonlinear heat flow problems \cite{douglas1956numerical} and later in the setting of maximal monotone operators \cite{lions1979splitting}. DR is widely used to solve the 2-set feasibility problem: 
$$
{\rm Find}\quad x \in A \cap B.
$$
Given two closed sets A, B and an initial point $x_0$, the DR is defined as the sequence $\{x_n\}_{n=1}^\infty$ generated by,
\begin{equation}
    x_{n+1}\in T_{A,B}(x_n)\quad\text{where}\quad T_{A,B}:= \frac{1}{2}(I+R_B\circ R_A).
\end{equation}

\begin{proposition}[Douglas--Rachford fixed point property {\cite{lions1979splitting}}]
    Let A and B be closed, convex sets, where $A\cap B \neq \emptyset$. Given any $x_0\in \mathcal{H}$, $\{x_{n}\}$ converges weakly to some point $x^* \in FixT$ and $\{P_A(x_n)\}$ is weakly convergent to $P_A(x^*)\in A\cap B$. 
    \label{DR fixed}
\end{proposition}

While the nonconvex setting is more complicated, Douglas--Rachford method exhibits a tendency to solve such problems with greater frequency that other projection-based solvers (e.g. alternating projections method). For this reason, reflection-based methods, and Douglas--Rachford in particular, have motivated the design of various other projection methods for solving nonconvex feasibility; these include CRM and the Lyapunov surrogate method.

\subsection{Lyapunov Surrogate Method}
    
    \begin{figure}
    \centering
    \includegraphics[width=0.5\linewidth]{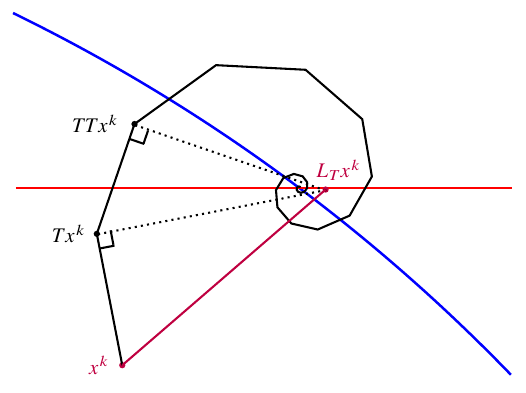}
    \resizebox{0.5\linewidth}{!}{

    }
    \caption{A LT update when applied to a circle and line in $\mathbb{R}^2$.}
    \label{fig:LT_diag}
\end{figure}

A method for computing the Lyapunov surrogate update for two closed sets $A$ and $B$ with an initial point $y^0$ is presented in \cite{calcan2024admm}. Let $v^0 \gets y^n$, and define successive iterates $v^1 \gets T_{A,B}(v^0)$ and $v^2 \gets T_{A,B}(v^1)$. Let $u^1=v^1-v^0$ and $u^2=v^2-v^0$. If the determinant
\begin{align*}
    \eta =\|u^1\|^2\|u^2\|^2-(\langle u^1,u^2\rangle)^2
\end{align*}
is zero, then the points $v^0,\ v^1,\ v^2$ are collinear and the construction is not possible (in which case one proceeds as in the otherwise case in \eqref{LT def} below). Otherwise define $u:= v^0+\mu_1u^1+\mu_2u^2$, where 
\begin{equation*}
   \begin{bmatrix}
        \mu_1 \\
        \mu_2
    \end{bmatrix}
    := \frac{1}{\eta}
    \begin{bmatrix}
    \|u^2\|^2 & -\langle u^1,u^2\rangle \\
    -\langle u^1,u^2\rangle & \|u^1\|^2
    \end{bmatrix}
    \begin{bmatrix}
    \|u^1\|^2 \\
    \|u^2\|^2-\langle u^1,u^2\rangle + \|u^1\|^2.
    \end{bmatrix}
\end{equation*}
The resulting point $u$ is the unique point in the affine subspace spanned by $\{v^0,v^1,v^2\}$ that satisfies the conditions
\begin{equation*}
    \langle u-v^1, v^1-v^0 \rangle = \langle u-v^2, v^2-v^1 \rangle = 0.
\end{equation*}
This construction is illustrated in Figure~\ref{fig:LT_diag} in the case of a circle and a line in $\mathbb{R}^2$ (see Section~\ref{subsec:circle_line} for numerical experiments on this example). We can then define the Lyapunov surrogate method as the sequence $\{y^n\}_{n=1}^\infty$ generated by
\begin{equation}
    y^{n+1}= L_{T_{A,B}}(y^n)\quad \text{where} \quad L_{T_{A,B}}(y^n):= \begin{cases}
        u& \text{if $v^0,v^1, v^2$ are not collinear;}\\
        v^1& \text{otherwise}.
        \end{cases}
    \label{LT def}
\end{equation}

\section{Quadratic Convergence}

In this section, we will prove local quadratic convergence of the Lyapunov surrogate method in the following setting.
\begin{enumerate}[label=\textbf{A\arabic*}, align=left, left=5pt, start=1]
    \item\label{A1} $A := \mathbb{R} \times \{0\}$ and $B := {\rm gra}f := \{(t,s) \mid s=f(t)\}$ where $f$ is real-analytic with $f(0)=0$ and $f'(0) \neq 0$.
\end{enumerate}
Let $y^0 \in \mathbb{R}^2$, and $w^0 := T_{A,B}^2y^0$. Whereas \cite{dao2019lyapunov}[Corollary 4.4] establishes linear convergence of the Douglas--Rachford sequence $(y^n)_n$ to $y^*=0$ in this context (with rate $1/\sqrt{1+|f'(0)|^2}$), we have that 
\begin{equation}\label{eqn:linear_rate}
    (\exists r_0 >0 , \eta >1) \quad \text{such that}\quad  y \in \mathcal{B}_{r_0}(0) \implies d(y^0,0) \geq \eta d(T^2y^0,0).
\end{equation}
We will therefore assume:
\begin{enumerate}[label=\textbf{A\arabic*}, align=left, left=5pt, start=2]
    \item\label{A2} $y^0 \in \mathcal{B}_{r_0}(0)$.
\end{enumerate}
Moreover, \cite[Corollary 3.9]{dao2019lyapunov} shows that for some $r_1>0$, $T_{A,B}$ is locally invertible on $\mathcal{B}_{r_1}(0) \times \mathbb{R}$, wherefore we also take the assumption:
\begin{enumerate}[label=\textbf{A\arabic*}, align=left, left=5pt, start=3]
    \item\label{A3} $y^0 \in \mathcal{B}_{r_1}(0)$.
\end{enumerate}
Combined with \ref{A2}, this guarantees that $w^0 \in \mathcal{B}_r(0)$, wherefore the first coordinate of $w^0$ has absolute value less than $r_1$, wherefore $T_{A,B}y^0 = T_{A,B}^{-1}w^0$.

\begin{lemma}\label{lem:ratio_complete}
    It holds that
    $$
    \underset{y \rightarrow 0}{\liminf}\frac{\|T^2y\|^2}{\|L_{T_{A,B}}y\|_1} >0.
    $$
\end{lemma}

\begin{proof}
    Appendix~A is devoted to the proof of this lemma.
\end{proof}


\begin{theorem}\label{thm:quadratic}
    Let the sets \(A,B\) be as in assumption~\ref{A1}, and \(r_0,r_1\) be as in~\ref{A2},~\ref{A3}. Let $y^{n+1}:=L_{T_{A,B}}y^n$. Then \(\exists r_2,r_3>0\) such that if \(y^0 \in \mathcal{B}_{\min \{r_0,r_1,r_2,r_3\}}(0)\), then \((y^n)_n\) converges to $0$ with rate that is never worse than 1/2-linear and ultimately quadratic. 
\end{theorem}


\begin{proof}
    For any $y^n$, denote $w^n := T^2y^n$. Recalling \eqref{eqn:linear_rate}, it holds that
	\begin{align}
	(y^n \in \mathcal{B}_{r_0}(0)) \quad \implies \quad d(y^n,0) \geq \eta d(w^n,0)
	\end{align}
	Now, from~Lemma~\ref{lem:ratio_complete}, we have that there exists $r_2 > 0$ such that 
	\begin{align}
	y^n \in \mathcal{B}_{r_2}(0) &\implies  \frac{d(w^n,0)^2}{d(y^{n+1},0)}> \frac{M}2\nonumber \\
	&\iff d(w^n,0)^2\left( \frac{2}{M}\right)> d(y^{n+1},0)\nonumber \\
	&\implies d(y^n,0)^2 \frac{2}{M\eta^2} > d(y^{n+1},0).\label{eqn:quadratic_rate}
	\end{align}
    Choose $r_3 = \min\{r_2, M\eta^2/4 \}$, and we have that 
	\begin{align*}
		y^n \in \mathcal{B}_{r_3}(0) &\implies \frac{M\eta^2}{4}d(y^n,0) \geq d(y^n,0)^2 \geq \frac{M\eta^2}{2}d(y^{n+1},0)\\
		&\implies \frac{1}{2}d(y^n,0) \geq d(y^{n+1},0).
	\end{align*}
Hence, letting $y^0 \in \mathcal{B}_{\min \{r_0,r_1,r_2,r_3 \} }$, we have that $(y^n)_{n}$ converges to zero with rate that is never worse than 1/2-linear and eventually quadratic by  \eqref{eqn:quadratic_rate}.
\end{proof}

\noindent From now on, we alias $L_T:=L_{T_{A,B}}$, which is certainly shorter.

\section{Numerical Results}
\label{sec:results}
In this section we present numerical experiments illustrating the local convergence behaviour of the Douglas--Rachford method and the Lyapunov surrogate method for two plane curve examples. We report two complementary performance measures: 
\begin{enumerate}[label=(\roman*)]
    \item Dolan--Moore performance profiles \cite{dolan2002benchmarking} across 1000 trials; and
    \item convergence profiles in terms of log--scale error $\log\|x^k-x^*\|$, where $x^*$ is the fixed point.
\end{enumerate}
Since one $L_T$ step requires twice the operator evaluations of a DR step, we report both iteration counts and CPU times. All experiments were implemented in \texttt{Python}, using the \texttt{mpmath} library to perform arbitrary--precision arithmetic. This high precision is necessary to reliably observe convergence behaviour. Implementation for all our results can be seen at \url{https://github.com/Jordancollard/quadratic-convergence-projection-plane-curve}.


\subsection{Circle and Line}
\label{subsec:circle_line}
We consider the problem of finding a point in the intersection of the line $L=\{x\in \mathbb{R}^2:x_2=0.5\}$ and the unit circle $\mathbb{S}^1=\{x\in \mathbb{R}^2:\|x\|=1\}$, a simple nonconvex feasibility problem that has been extensively studied because of its geometric similarity to phase retrieval \cite{aragon2013global,benoist2015sphere,borwein2018dynamics,borwein2011douglas}. The Douglas--Rachford operator used here is defined as $T_{L,\mathbb{S}^1}:=\frac{1}{2}(I+R_{\mathbb{S}^1}R_L)$. The projection onto L is given by
\begin{align*}
    P_L(x_1,x_2)=(x_1,0.5),
\end{align*}
and the projection onto the unit circle is 
\begin{align*}
P_{\mathbb{S}^1}(x) =
\begin{cases}
\dfrac{x}{\|x\|}, & \text{if } x \neq 0, \\
(1, 0), & \text{if } x = 0,
\end{cases}
\end{align*}
where we select the point $(1,0)$ since the projection is set valued at $x=0$. Initial points are generated within a radius of 0.5 around the intersection point $(\frac{\sqrt{3}}{2},0.5)$ to assess local convergence behaviour.

\begin{figure}[tb]
\centering
\begin{minipage}[t]{0.48\textwidth}
    \centering
    \includegraphics[width=\linewidth]{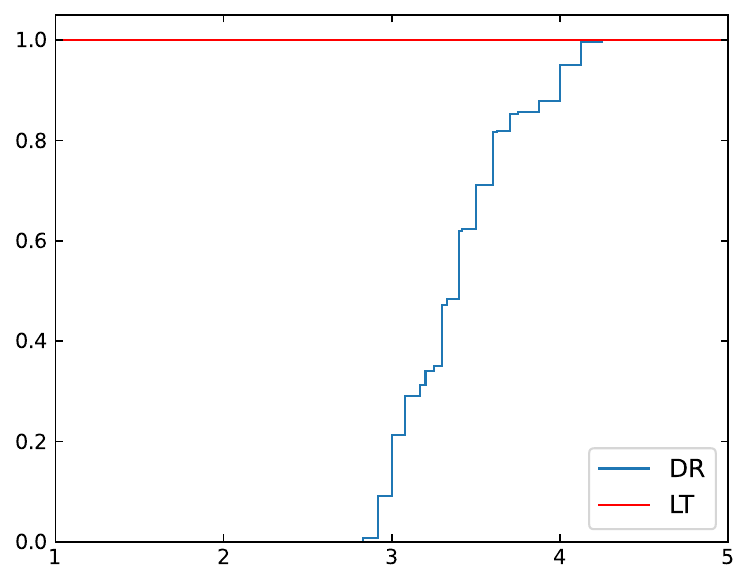}
    
\end{minipage}
\hfill
\begin{minipage}[t]{0.48\textwidth}
    \centering
    \includegraphics[width=\linewidth]{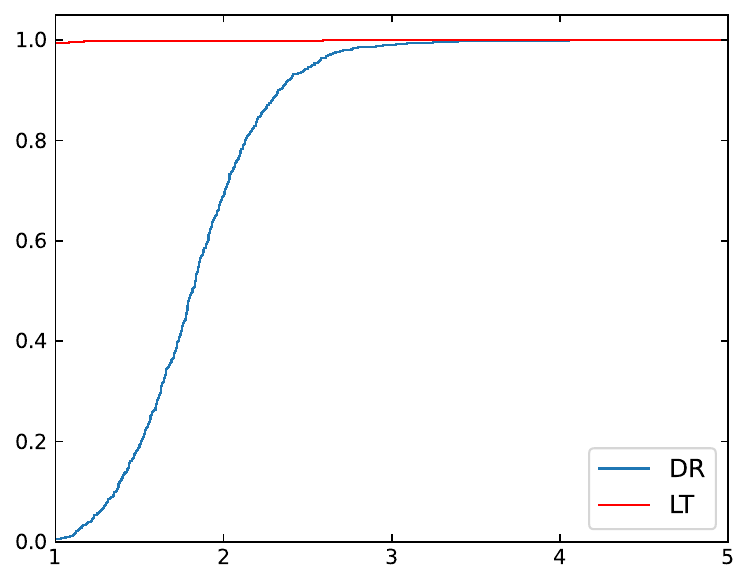}
    
\end{minipage}
\caption{Performance profiles comparing the Douglas--Rachford method and Lyapunov surrogate method for a circle and a line in $\mathbb{R}^2$.
\textbf{a} Based on number of iterates.
\textbf{b} Based on CPU time.
}
\label{fig:performance_profiles_circle_line}
\end{figure}

\begin{figure}[tb]
\includegraphics[scale=.6]{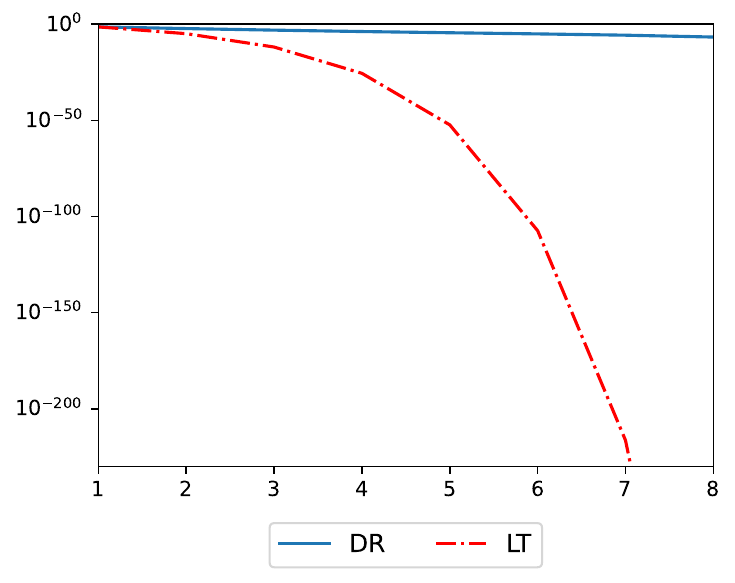}
\caption{Observed convergence rates of the Douglas--Rachford method and Lyapunov surrogate method for a circle and a line in $\mathbb{R}^2$.}
\label{fig:convergence plot_circle_line}       
\end{figure}
Figure~\ref{fig:performance_profiles_circle_line} demonstrates that LT consistently outperforms DR for this problem. In Figure~\ref{fig:convergence plot_circle_line} we see log-scale plot of the error, with DR converging linearly, and LT exhibiting quadratic convergence, in agreement with Theorem~\ref{thm:quadratic}.

\subsection{Semidefinite Feasibility}
\label{subsec:semidefinite_feasibility}
\begin{figure}[tb]
\includegraphics[scale=.6]{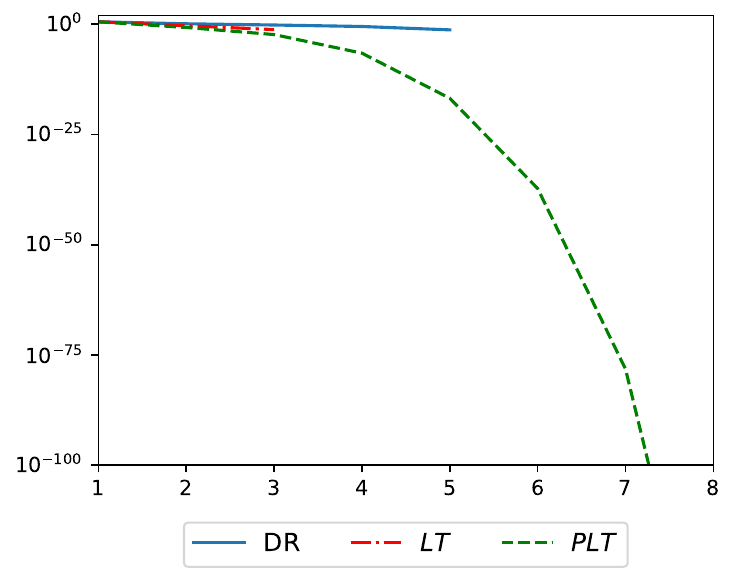}
\caption{Observed convergence behaviour of DR, LT, and PLT for Setting 1.}
\label{fig:convergence plot_psd_s1}       
\end{figure}
We next consider feasibility problems in the space of real symmetric matrices
\begin{equation*}
    S^n:=\{X\in \mathbb{R}^{n\times n}:X=X^\top\},
\end{equation*}
which arise in semidefinite relaxations of combinatorial optimisation problems, such as the Max-Cut problem \cite{BenTalNemirovski2001, goemans1995improved}. These problems can be seen as higher dimensional analogues of plane curve problems, now involving the nonlinear constraint of positive semidefiniteness and affine constraints on matrix entries. In all cases, we take $A$ to denote the nonlinear set and $B$ to denote the affine constraint set, so that the Douglas--Rachford operator is given by $T_{B,A}:=\tfrac{1}{2}(I+R_AR_B)$. Initial points are generated as symmetric matrices with entries sampled uniformly from $[-1,1]$ and then symmetrised. 

In this section, we also include the \emph{projected} Lyapunov surrogate method (PLT), defined by 
\begin{equation*}
    y^{k+1}=L_T\! \big(P_B(y^k)\big),
\end{equation*}
that is, applying $L_T$ after a projection onto the affine set. Initially introduced as a stabilisation technique for numerical error \cite{dizon2022circumcentering}, PLT defines a distinct algorithm with its own convergence behaviour, which we highlight in the experiments below.

\textbf{Setting 1}: We begin with the problem of finding a point in the intersection of the positive semidefinite (PSD) cone
\begin{equation*}
        S^n_+:=\{X\in S^n:v^\top Xv \geq 0 \text{ for all }  v\in \mathbb{R}^n\} 
\end{equation*}
and the affine subspace
\[
    S_1:=\{X\in S^n:\operatorname{diag} (X)=1\}.
\]
The projections onto these sets are
\begin{align}
    P_{S^n_+}(X) &= Q \,\operatorname{diag}(\max\{\lambda_i,0\})\, Q^\top, \label{eq:psd-proj} \\
    (P_{S_1}(X))_{ij} &=
    \begin{cases}
        \tfrac{1}{2}(X_{ij}+X_{ji}), & i \neq j, \\
        1, & i=j,
    \end{cases} \label{eq:s1-proj}
\end{align}
where $X = Q \operatorname{diag}(\lambda_1,\dots,\lambda_n) Q^\top$ is the spectral decomposition. The formula for \eqref{eq:psd-proj} is the standard eigenvalue-thresholding projection onto $S^n_+$ (see \cite{bauschke2011convex}[Example~28.25]); we provide a short, self-contained proof of \eqref{eq:s1-proj} below.
\begin{proof}
    We seek $P_{S_1}(X)=\arg\min_{Y\in S_1}\|Y-X\|^2_F.$
    Rewriting the objective
    \[
        \|Y-X\|^2_F=\sum_{i=1}^n(Y_{ii}-X_{ii})^2+\sum_{i<j}[(Y_{ij}-X_{ij})^2+(Y_{ji}-X_{ji})^2].
    \]
    Subject to $Y\in S_1$, the constraints are:
    \begin{enumerate}[label=(\roman*)]
        \item $Y_{ii}=1$ for all $i;$
        \item $Y_{ij}=Y_{ji}$ for all $i\neq j.$
    \end{enumerate}
    So for the diagonal entries we have $Y_{ii}=1$ fixed, so we don't need to optimise.
    For each off-diagonal pair $(i,j)$ with each $i<j$, let $y:=Y_{ij}=Y_{ji}$ and minimise
    \[
        \phi(y)=(y-X_{ij})^2+(y-X_{ji})^2.
    \]
    This is strictly convex with derivative
    \[
        \phi'(y)=2(y-X_{ij})+2(y-X_{ji})=4y-2(X_{ij}+X_{ji}),
    \]
    which when we set to zero and solve for y we get
    \[
        y=\frac{X_{ij}+X_{ji}}{2}.
    \]
    Putting these together gives 
    \[
        (P_{S_1}(X))_{ij} =
    \begin{cases}
        \tfrac{1}{2}(X_{ij}+X_{ji}), & i \neq j, \\
        1, & i=j,
    \end{cases}
    \]
    as required.
\end{proof}

Figure~\ref{fig:convergence plot_psd_s1} indicates that, in this setting, both DR and LT appear to terminate after finitely many steps, although we have no theoretical explanation for why finite convergence should occur here. In contrast, PLT does not terminate finitely and instead shows behaviour consistent with quadratic convergence, illuminating a distinction between LT and its projected variant. This clearly warrants further investigation.

\textbf{Setting 2}:
We next consider the intersection of the boundary of the PSD cone
\begin{equation*}
    \partial S^n_+:=\{X \in S^n_+ : \lambda_{\min}(X) = 0\}
\end{equation*}
with the affine subspace $S_1$. When we replace the PSD cone with its boundary, Slater's condition no longer holds, and we do not expect finite convergence to persist. The projection onto $\partial S^n_+$ is given by
\begin{align*}
    P_{\partial S^n_+}(X) =
    \begin{cases}
    P_{S^n_+}(X), & \lambda_{\min} \leq 0, \\[0.6ex]
    Q\,\mathrm{diag}(\mu)\,Q^\top, & \lambda_{\min} > 0.
\end{cases}
\end{align*}
where \(\mu\) is obtained from $(\lambda_1,\dots,\lambda_n)$ by replacing one smallest eigenvalue with $0$ and leaving all other eigenvalues unchanged. 
\begin{proof}
    Let $X\in S(n)$ with spectral decomposition 
    \[
        X=Q\operatorname{diag}(\lambda_1,\dots,\lambda_n)Q^\top,\quad \lambda_1\leq\dots\leq\lambda_n.
    \]
    Consider the nearest point problem on the boundary
    \begin{equation*}
        \min_{Y} \tfrac{1}{2}\|Y-X\|^2_F\quad \text{s.t. } Y\succeq0,\  \lambda_{\min}(Y)=0.
    \end{equation*}
    The Frobenius norm is unitarily invariant \cite{Horn_Johnson_2012}[Chapter~5.6] and $S^n_+$ is orthogonally invariant \cite{Horn_Johnson_2012}[Chapter~4.5]; therefore any nearest point in $S^n_+$ can be taken to share an orthonormal eigenbasis with $X$. Write
    \[
        Y=Q\operatorname{diag}(\mu_1,\dots,\mu_n)Q^\top\quad \text{with}\quad \mu_i\geq 0, \min_i\mu_i=0.
    \]
    Then
    \[
        \|Y-X\|^2_F=\|Q\operatorname{diag}(\mu_i-\lambda_i)Q^\top\|^2_F=\sum_{i=1}^n(\mu_i-\lambda_i)^2,
    \]
    where each $\mu_i$ can be optimised independently.
    \begin{enumerate}[label=(\roman*)]
    \item If $\min_i\lambda_i\leq0$, we have $\mu_i=\max\{\lambda_i,0\}$ minimises each term and satisfies $\min_i\mu_i=0.$ This is the usual PSD cone projection.
    \item If $\min_i\lambda_i>0$, the constraint forces at least one $\mu_i=0.$ The minimiser is then to set $\mu_i=\lambda_i$ for all but one index and sets $\mu_j=0$ for some $j$. Choosing $j\in \arg\min_i\lambda_i$ minimises the increase $\sum_i(\mu_i-\lambda_i)^2$. 
    \end{enumerate}
    Notice the projector onto \(\partial S^n_+\) may be set-valued when the smallest eigenvalue has multiplicity greater than one. In this case we fix a selection, denoted \(\Pi_{\partial}\), by zeroing the first index attaining \(\lambda_{\min}(X)\).
\end{proof}

\begin{figure}[tb]
\includegraphics[scale=.6]{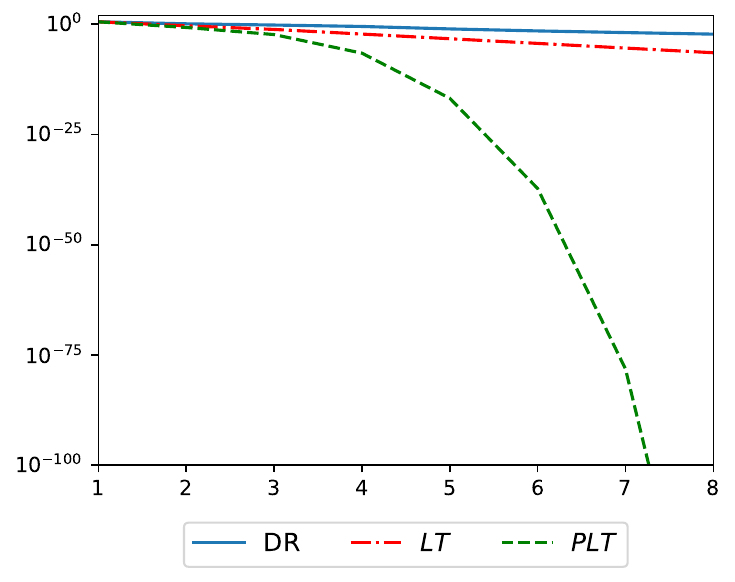}
\caption{Observed convergence behaviour of DR, LT, and PLT for Setting 2.}
\label{fig:convergence plot_psdb_s1}       
\end{figure}
As shown in Figure~\ref{fig:convergence plot_psdb_s1}, DR and LT no longer terminate finitely and instead appear to converge linearly, while PLT demonstrates convergence consistent with a quadratic rate.

\textbf{Setting 3}:
Lastly, we consider the intersection of the PSD boundary with the hyperplane
\[
    S_{11}:=\{X\in S^n:X_{11}=1\}.
\]
This case is more analogous to the setting in Theorem~\ref{thm:quadratic}, since $S_{11}$ is a hyperplane, the prototypical affine set in our framework. The projection onto $S_{11}$ is
\begin{equation}
     (P_{S_{11}}(X))_{ij} =
  \begin{cases}
    \tfrac{1}{2}(X_{ij}+X_{ji}), & i \neq j, \\
    1, & i=j=1, \\
    X_{ij}, & \text{otherwise}.
  \end{cases} \label{eq:s11-proj}
\end{equation}
The proof follows from the argument for \eqref{eq:s1-proj} with the only change being that only the $(1,1)$ entry is fixed.
\begin{figure}[tb]
\includegraphics[scale=.6]{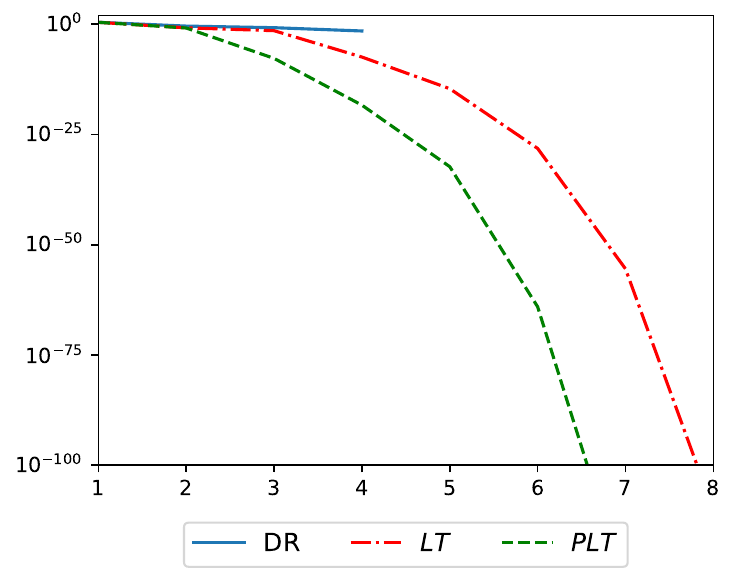}
\caption{Observed convergence behaviour of DR, LT, and PLT for Setting 3.}
\label{fig:convergence plot_psdb_s11}       
\end{figure}

As shown in Figure~\ref{fig:convergence plot_psdb_s11}, DR appears to converge finitely in this case, although we again we have no theoretical explanation for why finite convergence should occur here. In comparison, both LT and PLT exhibit behaviour consistent with quadratic convergence.

\section{Conclusion}
\label{sec:conclusion}
We established a theoretical guarantee of local quadratic convergence for the Lyapunov surrogate method when applied to the intersection of a smooth curve and a line in \(\mathbb{R}^2\) at a point where they intersect non-tangentially. This result provides the first analytical evidence supporting previously observed numerical behaviour. Complementing this, we presented numerical experiments on semidefinite feasibility problems, which serve as higher-dimensional analogues of plane curve intersections, showing the potential for quadratic convergence of Lyapunov surrogate methods beyond two dimensions. Future work should extend this analysis to higher-dimensional settings. Since this first result makes use of a characterisation \cite[Theorem~4.1]{lindstrom2022computable} that does not readily extend from 2 to $n$ dimensions, we expect such an extension to be non-trivial. Another direction is to better understand the role of projected LT and its impact on convergence rates in nonconvex feasibility.

\appendix

\section*{Appendix A: Characterisation of the $L_T$ Operator and Supporting Lemmas for Theorem~\ref{thm:quadratic}}
\label{sec:appendix}

Letting $w=T^2y$ with $y \in \mathcal{B}_{\min \{r_0,r_1 \}}(0)$, we will next describe the  Lyapunov surrogate minimising update $L_T(y)$. As the coordinates of $w$, in particular, will serve as the basis for our proof, we let $w=(x,z)$. Now \cite[Corollary 3.9]{dao2019lyapunov} furnishes a local expression for the inverse
\begin{align}\label{eqn:TinverseA}
	T_{A,B}^{-1}w &= \big(x + z f'(x), z - f(x)  \big).
\end{align}
Now \cite[Section~5]{dao2019lyapunov} furnishes a closed form strict-descent Lyapunov function $V$ whose gradient is as follows: $$
\nabla V(x,z) = \left(\frac{f(x)}{f'(x)},z \right).
$$ 
We also have from \cite[Theorem~4.1]{lindstrom2022computable} that for some $\gamma_1,\gamma_2 \in \mathbb{R}$, the $L_Ty$ update can be expressed\footnote{Note that this construction, corresponding to the non-collinear case in \eqref{LT def}, is defined whenever the denominator of the large rational term in \eqref{eqn:large_rational_term} is nonzero, and Lemma~\ref{lem:bigratio} shows that this is always the case locally near $0$, wherefore the collinear case in \eqref{LT def} cannot occur for $y$ sufficiently near to zero and may be ignored.} in two equivalent ways:
\begin{align*}
	L_T y &= w - \gamma_1 \nabla V(w) = (x,z)-\gamma_1 \left(\frac{f(x)}{f'(x)},z \right)\\
	\text{and}\quad L_T y &= T^{-1}w - \gamma_2 \nabla V(T^{-1}(w)).
\end{align*}
Substituting the explicit form of \(T^{-1}\) from \eqref{eqn:TinverseA} into the second expression yields
\begin{align*}
    L_Ty=\big(x+zf'(x),z-f(x)\big)-\gamma_2\bigg(\frac{f(x+zf'(x))}{f'(x+zf'(x))},z-f(x)\bigg).
\end{align*}
Thus we obtain two simultaneous expressions for $L_Ty$, leading to two equations in the two unknowns $\gamma_1,\gamma_2$. Collecting terms gives the linear system
\begin{align*}
	\left[\begin{array}{cc}
		-\frac{f(x)}{f'(x)} & \frac{f(x+zf'(x))}{f'(x+zf'(x))}\\
		-z & z-f(x)
		\end{array} \right]%
	\left[\begin{array}{c}
		\gamma_1\\ \gamma_2
	\end{array} \right] =%
	\left[\begin{array}{c}
		zf'(x)\\ -f(x)
	\end{array} \right].
\end{align*}
Inverting the left matrix yields explicit formulas for the coefficients
\begin{align}\label{eqn:large_rational_term}
	\left[\begin{array}{c}
		\gamma_1\\ \gamma_2
	\end{array} \right] &=%
	\frac{1}{-\frac{(z-f(x))f(x)}{f'(x)} + \frac{zf(x+zf'(x))}{f'(x+zf'(x))}}%
	\left[\begin{array}{c c}
		z-f(x) & -\frac{f(x+zf'(x))}{f'(x+zf'(x))} \\
		z & -\frac{f(x)}{f'(x)}
	\end{array} \right]%
	\left[\begin{array}{c}
		zf'(x) \\ -f(x)
	\end{array} \right].
\end{align}
In particular, notice that we can express $L_Ty$ directly in terms of the entries of $w=(x,z)$:
\begin{subequations}
\begin{align}
	\gamma_1 &= h(x,z) := \frac{ (z-f(x))zf'(x) + \frac{f(x)f(x+zf'(x))}{f'(x+zf'(x))}}{-\frac{f(x)(z-f(x))}{f'(x)} + \frac{zf(x+zf'(x))}{f'(x+zf'(x))}}\\
	(L_Ty)_1 &= x - h(x,z) \frac{f(x)}{f'(x)} \label{LTX}\\
	(L_Ty)_2 &= z- h(x,z)z \label{LTZ}
\end{align}
\end{subequations}
We now prepare a local Taylor expansion for $f$ and $f'$ about $t=0$, which will be used in the subsequent lemmas. Denote $f'(0) =a \neq 0$. Then $f$ and $f'$ admit the expansions
\begin{subequations}\label{eqn:Taylorseries}
\begin{align*}
	f(t) &= 0 + at + t^2b(t)\\
	f'(t) & = a + tc(t)
\end{align*}
\end{subequations}
Here $t^2b(t)$ is the tail of the series for $f(t)$, consisting of the degree 2 term and all following terms. Here $b(t)$ is the tail divided by $t^2$. Here $tc(t)$ is the tail of the series for $f'(t)$ consisting of the degree 1 term and all following terms, and $c(t)$ is the tail divided by $t$.

In order to prove Theorem~\ref{thm:quadratic}, it remains to establish the estimate
\begin{equation}\label{eqn:WTS}
    \liminf_{y \rightarrow 0} \frac{d(w,0)^2}{d(L_Ty,0)} >0.
\end{equation}
For the denominator, we will use the 1-norm distance,
$$
d(L_Ty,0) = |(L_Ty)_1| + |(L_Ty)_2|
$$
where $(L_Ty)_1,(L_Ty)_2$ have the closed form expressions in \eqref{LTX} and \eqref{LTZ}. We will also interchangeably use the Euclidean and polar coordinates of $w$, namely $(x,z)=(R\cos(\theta),R\sin(\theta)).$

\begin{lemma}\label{lem:zeta}
	It holds that
	\begin{align*}
		\underset{R\rightarrow 0}{\limsup}&\left(\frac{f(x+zf'(x))}{f'(x+zf'(x))}-zf'(x) - \frac{f(x)}{f'(x)}\right)/R^2\\
		&\leq \frac{\bar{\nu}(\theta)}{a^2}
	\end{align*}
	where the number $\bar{\nu}(\theta)$ is finite.
\end{lemma}
\begin{proof}
	We first put everything over a common denominator:
	\begin{align}
		&\frac{f(x+zf'(x))}{f'(x+zf'(x))}-zf'(x) - \frac{f(x)}{f'(x)}\nonumber\\
		=&\frac{f(x+zf'(x))f'(x)-zf'(x)^2f'(x+zf'(x)) - f(x)f'(x+zf'(x))}{f'(x+zf'(x))f'(x)}\nonumber\\
		=&\frac{\overbrace{f(x+zf'(x))f'(x)}^{=:\zeta_1}-\overbrace{zf'(x)^2f'(x+zf'(x))}^{=:\zeta_2} - \overbrace{f(x)f'(x+zf'(x))}^{=:\zeta_3}}{a^2+O(R)}.\label{eqn:3-zetas}
	\end{align}
	First consider the term $\zeta_1$. We have 
	\begin{align*}
		f(x+zf'(x)) &= a(x+zf'(x)) + (x+zf'(x))^2b(x+zf'(x))\\
		&=a(x+z(a+xc(x))) + O(R^2)\\
		&=ax+a^2z + O(R^2)\\
		&=R(a\cos(\theta)+a^2\sin(\theta)) +O(R^2)
	\end{align*}
	Consequently,
	\begin{align*}
		\zeta_1 &= \big(R(a\cos(\theta)+a^2\sin(\theta)) +O(R^2)\big)f'(x)\\
		&=R(a^2\cos(\theta)+a^3\sin(\theta)) +O(R^2).
	\end{align*}
	Next consider the term $\zeta_2$. We have
	\begin{align*}
		\zeta_2 = zf'(x)^2 f'(x+zf'(x)) &= R\sin(\theta)(a+O(R))^2(a+O(R))\\
		&=R\big(a^3\sin(\theta) \big) + O(R^2)
	\end{align*}
	Finally, consider the term $\zeta_3$. We have
	\begin{align*}
		\zeta_3 = f(x)f'(x+zf'(x)) &=(ax+x^2b(x))\big( a+ (x+zf'(x))\; c(x+zf'(x)) \big)\\
		&=R\big(a^2\cos(\theta) \big)+ O(R^2)
	\end{align*}
	Combining all of the above, we have 
	$$
	\hspace{-1cm}\zeta_1-\zeta_2-\zeta_3 = O(R^2)
	$$
	Consequently,
	\begin{equation}
	\underset{R\rightarrow 0}{\limsup} \frac{\zeta_1-\zeta_2-\zeta_3}{R^2} = \underset{R\rightarrow 0}{\limsup} \frac{O(R^2)}{R^2} \leq \sup_{\theta \in \left[0,2\pi\right)} \nu(\theta)=: \bar{\nu}(\theta),\label{eqn:zeta}
	\end{equation}
where $\nu(\theta)$ is the expression (of $\theta$) that corresponds to the $R^2$ term of $\zeta_1-\zeta_2-\zeta_3$. This is enough to conclude the result. Note that it is a straightforward (albeit tedious) arithmetic exercise to explicitly verify 
$$
\nu(\theta)=a^{2} \sin\! \left(\theta \right) \left(-c\! \left(0\right)+b\! \left(0\right)\right) \left(\sin\! \left(\theta \right) a+2 \cos\! \left(\theta \right)\right).
	$$
    To do so, one simply replaces the $O(R^2)$ terms with explicit expressions; we have used the $O$ terms to simplify the exposition, and because the exact form of $\nu$ is not required for our objective of verifying quadratic convergence. The exact form of $\nu$ can also be verified effortlessly with the aid of a computer algebra system such as \textit{Maple}.
\end{proof}

\begin{lemma}\label{lem:denominator}
	\begin{align*}
		\underset{R \rightarrow 0}{\lim}\left.\left(\frac{zf(x+zf'(x))}{f'(x+zf'(x))} - \frac{f(x)(z-f(x))}{f'(x)}\right) \middle/ R^2 \right. = a
	\end{align*}
\end{lemma}
\begin{proof}
	We put all terms over a common denominator:
	\begin{align}
		&\frac{zf(x+zf'(x))}{f'(x+zf'(x))} - \frac{f(x)(z-f(x))}{f'(x)}\nonumber\\
		&=\frac{zf(x+zf'(x))f'(x)  -     f(x)(z-f(x))f'(x+zf'(x))}{f'(x+zf'(x))f'(x)}\nonumber\\
		&=\frac{\overbrace{zf(x+zf'(x))f'(x)}^{=z \zeta_1}  -     \overbrace{f(x)(z-f(x))f'(x+zf'(x))}^{=(z-f(x))\zeta_3} }{a^2+O(R)}\label{eqn:z1z3}
	\end{align}
	Using our expansion of $\zeta_1,\zeta_3$ above, we can compute
	\begin{align}
		\underset{R \rightarrow 0}{\lim} \frac{z \zeta_1 -(z-f(x)) \zeta_3}{R^2} & = \sin(\theta)\big(a^2 \cos(\theta) + a^3\sin(\theta)\big)\nonumber \\
		&- (\sin(\theta)- a\cos(\theta))\cos(\theta)a^2\nonumber \\
		&= a^3 \sin(\theta)^2 + a^3 \cos(\theta)^2\nonumber \\
		&= a^3.\label{eqn:a_cubed}
	\end{align}
Combining \eqref{eqn:z1z3} and \eqref{eqn:a_cubed}, we have the result.
\end{proof}

\begin{lemma}\label{lem:bigratio}
    \begin{equation*}
    \underset{R\rightarrow 0}{\limsup} \;
    \frac{
    \frac{f(x + zf'(x))}{f'(x + zf'(x))} - zf'(x) - \frac{f(x)}{f'(x)}
    }{
    \frac{z f(x + zf'(x))}{f'(x + zf'(x))} - \frac{f(x)(z - f(x))}{f'(x)}
    }
    \leq \frac{\bar{\nu}(\theta)}{a^3}
    \end{equation*}
\end{lemma}

\begin{proof}
	Lemma~\ref{lem:zeta} treats the limit supremum of the numerator term over $R^2$, while  Lemma~\ref{lem:denominator} treats the limit of the denominator term over $R^2$. Combining these two Lemmas yields the result.
\end{proof}

\begin{lemma}\label{lem:1-hlimit}
	The following holds:
	\begin{align*}
		\underset{R\rightarrow 0}{\limsup}&\frac{1-h(x,z)}{R} \leq \sup_{\theta \in \left[0,2\pi\right)}\frac{( \sin(\theta)-a\cos(\theta)) \bar{\nu}(\theta)}{a^3} =: \psi(\theta)
	\end{align*}
\end{lemma}
\begin{proof}
	We first put all terms over a common denominator:
	\begin{align*}
	1-h(x,z) = -\frac{\left((z-f(x))zf'(x) + \frac{f(x)f(x+zf'(x))}{f'(x+zf'(x))} + \frac{f(x)(z-f(x))}{f'(x)} - \frac{zf(x+zf'(x))}{f'(x+zf'(x))} \right)}{\frac{zf(x+zf'(x))}{f'(x+zf'(x))} - \frac{f(x)(z-f(x))}{f'(x)}}
	\end{align*}
	Noticing that the numerator's 2nd and 4th rational terms share a common factor, the numerator rewrites to
	\begin{align*}
		(z-f(x))zf'(x) + \frac{f(x+zf'(x))}{f'(x+zf'(x))}(f(x)-z) + \frac{f(x)(z-f(x))}{f'(x)}
	\end{align*}
    Further noticing that the first and last rational terms share a common factor, the numerator further rewrites to
	\begin{align*}
		&(z-f(x))\left(zf'(x) +\frac{f(x)}{f'(x)}\right) + \frac{f(x+zf'(x))}{f'(x+zf'(x))}(f(x)-z) \\
		=&(f(x)-z) \left(\frac{f(x+zf'(x))}{f'(x+zf'(x))} - zf'(x) -\frac{f(x)}{f'(x)}  \right)
	\end{align*}
	Altogether, we have that
	\begin{align}
		1-h(x,z) =\frac{-(f(x)-z) \left(\frac{f(x+zf'(x))}{f'(x+zf'(x))} - zf'(x) -\frac{f(x)}{f'(x)}  \right) }{ \left({\frac{zf(x+zf'(x))}{f'(x+zf'(x))} - \frac{f(x)(z-f(x))}{f'(x)}} \right)} \label{eqn:biglimita}
	\end{align}
	We also have that 
	\begin{align}
		f(x)-z = ax+x^2b(x)-z &= R(a\cos(\theta)+R\cos(\theta)^2b(R\cos(\theta))- \sin(\theta))\nonumber \\
        &=R(a\cos(\theta)-\sin(\theta))+O(R^2).
        \label{eqn:biglimitb}
	\end{align}
	Combining \eqref{eqn:biglimita}, \eqref{eqn:biglimitb}, and Lemma~\ref{lem:bigratio} yields the result.
\end{proof}

\begin{lemma}\label{lem:LT}
	The following equalities hold.
	\begin{align*}
		(L_T y)_2 & = R^2\sin(\theta)\frac{(1-h(x,z))}{R}; \quad \text{and}\\
		(L_T y)_1 
		&=R^2 \frac{\cos(\theta)}{a+xc(x)}\left(a\frac{1-h(x,z)}{R} + \cos(\theta)(c(x)-b(x)h(x,z)) \right).
	\end{align*}
\end{lemma}
\begin{proof}
	Since $(L_Ty)_2 = z-h(x,z)z = z(1-h(x,z))$, the correctness of the first statement is clear. To see the correctness of the second statement, simply observe
	\begin{align*}
		(L_T y)_1 &= x-h(x,z)\frac{f(x)}{f'(x)}\\
		&= x-h(x,z)\frac{ax+x^2b(x)}{a+xc(x)}\\
		&= x\left( 1-h(x,z)\frac{a+xb(x)}{a+xc(x)} \right)\\
		&=\frac{x}{a+xc(x)} \left(a+xc(x)-(a+x b(x))h(x,z) \right)\\
		&=\frac{x}{a+xc(x)}\left( a(1-h(x,z)) + x\left( c(x)-b(x)h(x,z)\right) \right)\\
		&=R^2 \frac{\cos(\theta)}{a+xc(x)}\left(a\frac{1-h(x,z)}{R} + \cos(\theta)(c(x)-b(x)h(x,z)) \right).
	\end{align*}
\end{proof}

\begin{proof}[Proof of Lemma~\ref{lem:ratio_complete}]
From Lemma~\ref{lem:LT}, we can express the ratio as
	\begin{align*}
	\frac{\|(x,z)\|^2}{\|L_Ty\|_1} &= \frac{R^2}{R^2 \left|\sin(\theta)\frac{1-h(x,z)}{R} \right|+R^2 \left|\frac{\cos(\theta)}{a+xc(x)}\left(a\frac{1-h(x,z)}{R} + \cos(\theta)(c(x)-b(x)h(x,z)) \right) \right|}\\
	&=\frac{1}{\left|\sin(\theta)\frac{1-h(x,z)}{R} \right|+ \left|\frac{\cos(\theta)}{a+xc(x)}\left(a\frac{1-h(x,z)}{R} + \cos(\theta)(c(x)-b(x)h(x,z)) \right) \right|}
	\end{align*}
Ergo, it suffices to show that the denominator of the latter term stays bounded as $R \downarrow 0$. Consider the first term. 
\begin{align*}
\limsup_{R\downarrow 0} \left |\sin(\theta) \frac{1-h(x,z)}{R}\right| \leq \limsup_{R\downarrow 0}\left| \frac{1-h(x,z)}{R}\right|
\stackrel{\text{(Lemma~\ref{lem:1-hlimit})}}{\leq} \psi(\theta)< \infty.
\end{align*}
Likewise, for the second denominator term, we again use Lemma~\ref{lem:1-hlimit} to obtain
\begin{align*}
    &\limsup_{R \downarrow 0} \left|\frac{\cos(\theta)}{a+xc(x)}\left(a\frac{1-h(x,z)}{R} + \cos(\theta)(c(x)-b(x)h(x,z)) \right) \right|\\
    &\leq  \left| \frac{1}{a} \right| \left( |a| \psi(\theta) + |c(0)-b(0)| \right) < \infty.
\end{align*}
This shows the result.
\end{proof}

\bibliographystyle{plain}
\bibliography{bibliography}

\end{document}